\documentclass[12pt]{article}
\usepackage{amsmath,amsthm,amssymb,amsfonts,color}

\newcommand{\Octoni}{{\mathbb{O}}}     
\newcommand{\R}{{\mathbb{R}}}          

\newcommand{\rr}{\rightarrow}
\newcommand{\lrr}{\longrightarrow}

\newcommand{\calG}{{\cal G}}             %
\newcommand{\calR}{{{\cal R}^U}}             %

\newcommand{\na}{{\nabla}}

\newcommand{\Aut}[1]{{\mathrm{Aut}}\,{#1}}

\newcommand{\dx}{{\mathrm{d}}}

\newcommand{\tile}{{\tilde{e}}}

\newcommand{\vol}{{\mathrm{vol}}}
\newcommand{\Vol}{{\mathrm{Vol}}}

\newcommand{\expo}{{\mathrm{e}}}

\newtheorem{teo}{Theorem}[section]

\newtheorem{coro}{Corollary}[section]
\newtheorem{prop}{Proposition}[section]

\def\cyclic{\mathop{\kern0.9ex{{+}
\kern-2.2ex\raise-.28ex\hbox{\Large\hbox{$\circlearrowright$}}}}\limits}

\title{Variations of gwistor space}

\author{R. Albuquerque\footnote{Departamento de Matem\'atica da Univer\-si\-dade de
\'Evora and Centro de Investiga\c c\~ao em Matem\'atica e Aplica\c c\~oes (CIMA-U\'E), Rua
Rom\~ao Ramalho, 59, 671-7000 \'Evora, Portugal.}\\ rpa@dmat.uevora.pt}

\begin{document}


\maketitle


\begin{abstract}

We study natural variations of the $G_2$ structure $\sigma_0\in\Lambda^3_+$ existing on
the unit tangent sphere bundle $SM$ of any oriented Riemannian 4-manifold $M$. We find a
circle of structures for which the induced metric is the usual one, the so-called Sasaki
metric, and prove how the original structure has a preferred role in the
theory. We deduce the equations of calibration and cocalibration, as well as those of
$W_3$ pure type and nearly-parallel type.

\end{abstract}


\vspace*{10mm}


\vspace*{2mm}


\vspace*{10mm}

The author acknowledges the support of Funda\c{c}\~{a}o Ci\^{e}ncia e Tecnologia, Portugal, through
CIMA-U\'E, Centro de Investiga\c c\~ao em Matem\'atica e Aplica\c c\~oes da Universidade de \'Evora,
and through SFRH/BSAS/895/2009 (sabbatical scholarship).

\markright{\sl\hfill  R. Albuquerque \hfill}

\vspace*{10mm}

\section{Introduction}

In \cite{AlbSal1,AlbSal2} it was shown how a natural
$G_2=\Aut{\Octoni}$ structure is associated to the unit tangent sphere bundle
$\pi:SM\rr M$ of any given oriented Riemannian 4-manifold $M$. The techniques
are twistorial, such as those learned by the author from \cite{Obri}, so we have
chosen to give the name of $G_2$-twistors or simply \textit{gwistors} to the new
spaces.

The theory starts by a construction of the octonions inside $TTM$, restricted to the
3-sphere fibre bundle $SM$, which we take a moment to explain.
Recall the Levi-Civita connection of the base induces a canonical splitting
of the tangent bundle of $TM$. Both vertical and horizontal subbundles $V,H$
become isometric to $\pi^*TM$ with the pull-back metric. The metric direct sum over
$TM$ is called the Sasaki metric of this manifold. Independently of the metric, $V$
has a tautological section, denoted $U$ and defined by $U_u=u$, hence also a vertical
vector field on $SM=\{u\in TM:\ \|u\|=1\}$. Now
each point $u$ is identified with the identity element or the generator of the real
line in $\Octoni$. Then we use the volume form coupled with $U$, to induce a
cross-product on $u^\perp\subset V$. A conjugation map is trivial to define. Together
these give a quaternionic structure on $V$. Then, applying the well-known
Cayley-Dickson process, we obtain the $\Octoni$-structure on $V\oplus H$. 

The pull-back of $TM$ also inherits a metric connection $\na^*=\pi^*\na$ and hence
parallel identifications of horizontals and verticals, passing through $\pi^*TM$, cf. loc.
cit. and \cite{Sakai}. The manifold $SM$ is endowed with the induced metric from the
canonical or Sasaki metric on $TM$. Clearly $TSM$ coincides with $V_1\oplus H$ where
$V_1=\{v\in V:\ \langle u,v\rangle=0\}$ at each point $u$. Since $u$ is pointing
outwards, our space $SM$ inherits a $G_2$-structure, for which it receives the name
of gwistor space. Recall $G_2=\Aut{\Octoni}$. Clearly the structure is the extension
of an $SO(3)$ structure. The connection induces a projection $\na^*_.U:TSM\rr V$ with
kernel $H$.

By a Theorem of Y.~Tashiro in \cite{Blair} it is known that $SM$ has an almost
contact structure for a base of arbitrary dimension. As these are rigid geometrical
objects, the contact structure is bound to be K-contact if and only if $M$ is of
constant
sectional curvature 1. Then it turns out also to be Sasakian. Locally the space is
the same as the trivial fibration $V_{5,2}=SO(5)/SO(3)$.

Now if we leave aside the Cayley-Dickson process and concentrate on the five invariant
3-forms which are naturally defined on $SM$, then we may try to find other interesting
$G_2$ structures. This article is devoted to them, the \textit{variations} of gwistor
space, which may also be called $g$-natural $G_2$-structures on the unit tangent sphere
bundle, in analogy with the terms for the metrics used by \cite{AbbaCal,AbbaKowa} and many
references therein. On the other hand, the terms deformation or perturbation are
also used in similar context by other authors, so we made a choice.

We readily announce the utilization of some computer algebra software for
the proof of Theorem \ref{cocalibrationingeneral} below. It is a polynomial
computation of the 7th order in four variables which we believe anyone can reproduce
easily.

This work was started during the author's sabbatical leave at Philipps Universit\"at, Marburg,
and only later finished at IHES, Bures-sur-Yvette. He kindly acknowledges the hospitality
of both institutions and expresses his thanks to I. Agricola, Th. Friedrich,
M. Kontsevich and S. Meinhardt for fruitful conversations.


\subsection{The basic 3-forms}

We start by abbreviating the notation and write $SM=\calG$. There is, as we have seen, an
isometry connecting $H$ with $V$, which we denote by $B$. We extend it by 0 to $V$
thus producing an endomorphism $B$ of $TTM$. Then the tangent vector field
$B^tU$ generates a real line bundle, contained in $T\calG$, and a 1-form
$\theta=(B^tU)^\flat$. We may write a splitting, with $H_1=B^tV_1$:
\begin{equation}\label{splittingdeTSM}
 T^*\calG=\R\theta\oplus H_1^*\oplus V_1^* .
\end{equation}
We now pass to the language of differential forms. The 1-form $\theta$ is the
aforementioned contact structure, satisfying:
\begin{equation}
 \theta_u(v)=\langle u,\dx\pi(v)\rangle,\quad\quad\forall u\in\calG,\ v\in T\calG.
\end{equation}
The usual pull-back (horizontal) of the volume form of $M$ is also denoted by
$\vol$. The vertical pull-back of $\vol\in\Omega^4(M)$ contracted with $U$ is denoted by
$\alpha$; then we define analogously a 3-form $\alpha_3=(B^tU)\lrcorner\vol$.
Of course,
\begin{equation}
 \theta\wedge\alpha_3=\vol,\qquad\qquad  \vol\wedge\alpha=\Vol_{\calG}.
\end{equation}
As shown in \cite{Alb2}, it is possible to find locally an `adapted' frame, i.e. an
oriented orthonormal frame $e_0,e_1,\ldots,e_6$ respecting \eqref{splittingdeTSM}. In
particular such that (with usual notation for the co-framing, $e^{ab\cdots
c}=e^a\wedge e^b\wedge\cdots\wedge e^c$)
\begin{equation}\label{algumasformas}
 \theta=e^0,\qquad\qquad\alpha_3=e^{123},\qquad\qquad\alpha=e^{456}.
\end{equation}
It is easy to compute that $\dx\theta(v,w)=\langle(B^t-B)v,w\rangle,\ \forall
v,w\in T\calG$, which restricts to a symplectic 2-form on the vector bundle $H_1\oplus
V_1$ and hence is written as $\dx\theta=e^{41}+e^{52}+e^{63}$.

The endomorphism $B$ allows one to construct two other 3-forms (see \cite{Alb2} for the
invariant definition):
\begin{equation}
 \alpha_1=e^{156}+e^{264}+e^{345}
\end{equation}
and
\begin{equation}
 \alpha_2=e^{126}+e^{234}+e^{315}.
\end{equation}
One can prove the five 3-forms $\alpha,\ldots,\alpha_3,\theta\wedge\dx\theta$ correspond
to a
basis for the space of invariants in $\Lambda^3(\R\oplus\R^3\oplus\R^3)$ under the action
of $SO(3)$, the underlying structure group of $\calG$, ie. there are five irreducible
1-dimensional submodules\footnote{The author acknowledges I. Agricola and Th. Friedrich
for this computation.}.

The five 3-forms satisfy the `first structure equations', 
$\forall i=1,2,3$:
\begin{equation}\label{bse1} 
\begin{split}
*\alpha=\theta\wedge\alpha_3=\vol=\pi^*\vol_M,\ \ \ \ \ \ \ \ \
*\alpha_1=-\theta\wedge\alpha_2,\ \ \ \ \ \ \ \
*\alpha_2=\theta\wedge\alpha_1,\hspace{0cm}\\
*\dx\theta=\frac{1}{2}\theta\wedge(\dx\theta)^2,\ \ \ \ \
*(\dx\theta)^2=2\theta\wedge\dx\theta,\ \ \ \ \
*(\dx\theta)^3=6\theta,\hspace{0cm} \\
\alpha_1\wedge\alpha_2=3\alpha_3\wedge\alpha=3*\theta=\frac{1}{2}(\dx\theta)^3,\ \
\ \ \
\dx\theta\wedge\alpha_i=\dx\theta\wedge*\alpha_i=\alpha_3\wedge\alpha_i=0, \\
\dx\theta\wedge\alpha=\dx\theta\wedge*\alpha=\alpha\wedge\alpha_1
=\alpha\wedge\alpha_2=0.
\end{split}
\end{equation}

The natural $G_2$ structure on $\calG$ to which we have referred {\em {is}} given\footnote{Actually
the structure was given first by the opposite, $-\sigma_0$, but we take the
opportunity here to make the change.} by the 3-form
\begin{equation}
 \sigma_0=\alpha_2-\alpha+\theta\wedge\dx\theta.
\end{equation}
This form gives the \emph{canonical} representation theory without changing the canonical
orientation of $\calG$; namely it gives the usual $G_2$-modules
$\Lambda^2_7,\Lambda^2_{14}$ (which appeared from opposite highest weights in
\cite{Alb2,Alb3,AlbSal1,AlbSal2}). 

The integrability of $\sigma_0$ was studied first in the
case of the Levi-Civita connection on $M$ (\cite{AlbSal1,AlbSal2}) and
later in the case of metric connections with torsion, which clearly
allow the same construction as the unique which is torsion free (\cite{Alb2}). For
the latter we know that the $G_2$-twistor structure is cocalibrated, ie.
$\dx*\sigma_0=0$, if and only if the base $M$ is an Einstein manifold.

\subsection{Variations of $G_2$ structures}
\label{VoG2s}

Let us recall the definition of stable forms from the theory of $G_2$-manifolds,
\cite{Bryant1,Bryant2}.

Let $\sigma$ denote a linear $G_2$ structure on a 7-dimensional oriented vector space $V$. A consequence of the study
of the Lie group $G_2=\Aut{\sigma}\subset SO(7)$ is that it is connected and 14 dimensional;
henceforth, that the orbit of $\sigma$ under $GL(7,\R)$ is an open set inside the module
$\Lambda^3V^*$. This orbit is denoted $\Lambda_\pm^3$ and known as the space of stable
$G_2$-structures on $V$. We detect the boundaries of such stability
by the non-degeneracy of the induced Euclidean product. Indeed, the inner product is given
by the clearly symmetric map $(v,w)\mapsto v\lrcorner\sigma\wedge
w\lrcorner\sigma\wedge\sigma$ --- with this image 7-form required to be, on the
diagonal of $V$, a positive multiple of the chosen orientation. The given $\sigma$
satisfies this condition by assumption. Letting also $\sigma$ vary, we do
have a $GL(7,\R)$-equivariant map 
\[  V\otimes V\otimes\Lambda^3V^*\lrr\Lambda^7V^* . \]
Since $\Lambda^7V^*\backslash0$ has two connected components, we conclude
$\Lambda_\pm^3$ is the union of two open orbits under the action of
the subgroup $GL^+(7,\R)$, identified bijectively by a $-$ sign because
$(-1)^3=(-1)^7$. Moreover, the orientation
in $V$ induced by the first map itself is preserved in each of \textit{these} orbits.
Next we shall be concerned only with the positive-definite side
$\Lambda_+^3$ of $\Lambda_\pm^3$. 

We remark furthermore on the existence of split-octonionic structures, with
automorphism group the non-compact dual form of $G_2$, this time inducing metrics of
signature $\pm(3,4)$. In the following applications to gwistor space we shall not be
worried with the parallel with the split-octonionic structures, since such study may
be more easily undertaken afterwards. 

We return to the gwistor space $\calG\rr M$ and consider a variation of the
standard structure $\sigma_0$. We let $f_0,\ldots,f_4$ be scalar functions on $\calG$
and define
\begin{equation}\label{sigmageral}
 \sigma=f_0\alpha+f_1\alpha_1+f_2\alpha_2+f_3\alpha_3+f_4\theta\wedge\dx\theta.
\end{equation}
The original $G_2$ structure $\sigma_0$ is given by $-f_0=f_2=f_4=1,\
f_1=f_3=0$. At least for sufficiently close values to the standard, we do
obtain new $G_2$-struc\-tures. For the fixed orientation $\Vol_{\calG}=e^{0\cdots6}$,
induced by the Sasaki structure on $TM$ and the vector field $U$, we have that on any
two vectors $v,w$:
\begin{equation}
 v\lrcorner\sigma\wedge w\lrcorner\sigma\wedge\sigma= 
6\langle v,w\rangle_\sigma\Vol_{\sigma}= 6\langle v,w\rangle_{\sigma}m\Vol_{\calG}.
\end{equation}
The second identity defines $m>0$ as a scalar function of $\sigma$, by linearity and
because, as explained, $\sigma$ determines both the metric and the volume form, given
the orientation. $m:\calG\rr\R$ is already assumed to be positive---as we may without
loss of regularity, if the $f_i$ are smooth, or significant generality of the same
set of functions.

Lengthy but easy computations yield the result which we present next.

The metric matrix of $\langle\,\cdot\,,\,\cdot\,\rangle_\sigma$ with respect to the
adapted frame is:
\begin{equation}
 [\langle e_i,e_j\rangle_\sigma]=t\left[\begin{array}{ccccccc}
 f_4^2& & & & & & \\ &x& & & z& & \\ & &x& & & z& \\ & & &x& & & z\\
 &z& & & y& & \\ & &z& & & y& \\ & & &z& & & y \end{array}\right]
\end{equation}
where we have simplified notation by writing
\begin{equation}\label{defsdetxyz}
t=\frac{f_4}{m},\qquad x=f_2^2-f_1f_3,\qquad y=f_1^2-f_0f_2,\qquad z=f_1f_2-f_0f_3.
\end{equation}
Notice that $\sigma_0$ corresponds to the identity $1_7$. Computing determinants, the
metric is positive-definite if $f_4>0$, $x>0$ and $xy-z^2>0$. This proves the
following result.
\begin{teo}\label{condicoesdeestabilidade}
If a set of scalar functions $f_0,\ldots,f_4$ induces a $G_2$ structure on $\calG$, then it satisfies
 $f_4>0,\ \,f_2^2-f_1f_3>0$ and
\begin{equation}\label{polinomiode4ordem}
 3f_0f_1f_2f_3-f_0f_2^3-f_0^2f_3^2-f_3f_1^3>0.
\end{equation}
\end{teo}
\begin{trivlist}
\item {\bf Remarks}. 1. The homogeneous fourth degree polynomial is irreducible and has no critical
values in the domain. 2. The metrics obtained are all natural metrics in the sense of
\cite{AbbaCal,AbbaKowa} and other references therein.
\end{trivlist}

Using the Gram-Schmidt process on the new metric, we obtain the oriented
or\-tho\-nor\-mal
frame, for $i=1,2,3$,
\begin{equation}\label{ortframenewmetric}
 \tile_0=\frac{1}{f_4\sqrt{t}}e_0,\qquad\tile_i=\frac{1}{\sqrt{tx}}e_i,\qquad\tile_{i+3}=
\sqrt{\frac{x}{th}}(e_{i+3}-\frac{z}{x}e_i),
\end{equation}
where $h$ is the polynomial in (\ref{polinomiode4ordem}):
\begin{equation}\label{defdeh}
 h=xy-z^2.
\end{equation}

A dual co-frame is then
\begin{equation}\label{ortcoframenewmetric}
 \tile^0=f_4\sqrt{t}e^0,\qquad\tile^i=\sqrt{tx}e^i+z\sqrt{\frac{t}{x}}e^{i+3},\qquad\tile^{i+3}
=\sqrt{\frac{th}{x}}e^{i+3}.
\end{equation}
We obtain also the useful formulas
\begin{equation}\label{ortcoframenewmetric1}
 e^0=\frac{1}{f_4\sqrt{t}}\tile^0,\qquad
e^i=\frac{1}{\sqrt{txh}}(\sqrt{h}\tile^i-z\tile^{i+3}),\qquad e^{i+3}=\sqrt{\frac{x}{th}}\tile^{i+3}.
\end{equation}
Indeed the frame (\ref{ortframenewmetric}) is oriented, ie.
$\tile^{0123456}=m\,e^{0123456}$ is a positive multiple of the chosen orientation.
Immediately through \eqref{defsdetxyz} and \eqref{ortcoframenewmetric} we find that
\begin{equation}\label{efecinco}
   m=f_4h^\frac{1}{3}.  
\end{equation}

\subsection{$G_2$-structures $\sigma$ compatible with the Sasaki metric}

Let $\sigma$ be a variation of $\sigma_0$.
\begin{prop}\label{sigmarespeitandoametrica}
 The metric induced by $\sigma$ coincides with the Sasaki metric on $\calG$ if and only if 
\begin{equation}\label{sigmarespeitandoametrica1}
  f_0^2+f_1^2=1,\qquad f_2=-f_0,\qquad  f_3=-f_1,\qquad f_4=1.  
\end{equation}
Under the action of $SO(7)$ the orbit of 3-forms which can be written in the
form (\ref{sigmageral}) is a circle $S^1$.
\end{prop}
\begin{proof}
By hypothesis, we have $tf_4^2=tx=ty=1$ and $z=0$. Hence $f_4^3=f_4x=f_4y=m$ and
$h=xy=f_4^4$. Knowing $m$ must equal 1 or equating through \eqref{efecinco} we get all
these equal to 1, except for $z$. Now solving the system (\ref{defsdetxyz}) we deduce the
equivalence in the first part of the result. The second follows from the first (as
the metric is preserved) and the analysis of the orbit of
$\sigma_0=\alpha_2-\alpha+\theta\wedge\dx\theta$ through known methods. So, we note
that  already $U(3)\subset SO(7)$ acts as a real group, fixing $e_0$,
on the vector space $E=H_1\oplus V_1$, which has a natural complex structure.
Moreover,
\begin{equation*}
\begin{split}
 (e^1+\sqrt{-1}e^4)\wedge(e^2+\sqrt{-1}e^5)\wedge(e^3+\sqrt{-1}e^6)
&= \\ \alpha_3-\alpha_1+\sqrt{-1}
(\alpha_2-\alpha)=:\eta \,\in\Lambda^3{E^{(1,0)}}^* & 
\end{split}
\end{equation*}
As $SU(3)\subset G_2$ we are left to consider maps $g$ such that
$g_{|E}=\expo^{is}1_E$ for some $s\in\R$. One finds easily the role of $g$ as a real
map. Immediately we deduce $g$ fixes the 3-form
$\theta\wedge\dx\theta=e^{041}+e^{052}+e^{063}$. On the other hand
$g\cdot\eta=g^3\eta$. Letting
$g$ be such that $g^3=f_0+\sqrt{-1}f_1\in S^1$ we find that this real map solves
($\Im$ denotes imaginary part)
\[ g\cdot\sigma_0=g\cdot(\Im\eta+\theta\wedge\dx\theta)
=\Im(g^3\eta)+\theta\wedge\dx\theta= 
-f_0\alpha-f_1\alpha_1+f_0\alpha_2+f_1\alpha_3+\theta\wedge\dx\theta. \]
The result follows (notice the space $SO(7)/G_2$ is 7 dimensional so 
we have to restrict our statement to the specific forms).
\end{proof}

For the following computations we apply formulas which have been deduced in \cite{Alb2,AlbSal1,AlbSal2}.
We start by the particular case found above, when the Sasaki metric is preserved.
\begin{teo}\label{caseoffixedmetric}
Suppose the Riemannian manifold $M$ is connected. Let $\sigma$ be a variation of
gwistor space satisfying the condition that the
induced metric coincides with the Sasaki metric on $\calG$, that is,
$\sigma=-f_0\alpha-f_1\alpha_1+f_0\alpha_2+f_1\alpha_3+\theta\wedge\dx\theta$ with
$(f_0,f_1):\calG\rr S^1$ a smooth function. Then we have:\\
1. Always $\dx\sigma\neq0$.\\
2. If $(f_0,f_1)\neq(\pm1,0)$, then $\dx*\sigma=0$ if and only if the functions $f_0,f_1$ are
constant and the Riemannian base $M$ has constant sectional curvature.\\
3. If $(f_0,f_1)=(\pm1,0)$, then $\dx*\sigma=0$ if and only if $M$ is Einstein.
\end{teo}
The proof follows by recalling the list of derivatives of the fundamental 3-forms in
(\ref{derivadasdosalphas}), which were deduced in \cite[Proposition 2.3]{Alb2}. Result (1) is the
particular case of Theorem \ref{dsigmanevervanishes} (below). For (2) we may easily
compute $\dx*\sigma$. If it is to vanish, then we deduce a curvature equation
$R_{0123}=0$, which implies constant sectional curvature on the base, and that 
$f_0\dx f_0=-f_1\dx f_1$ is a multiple of $\theta$, which implies $(f_0,f_1)$ is constant.
Finally, if the base metric has constant sectional curvature $k$, then another
curvature term appearing satisfies
$\calR\alpha=-k\theta\wedge\alpha_1$, and we find this is the solution required in
case $f_1\neq0$.

Theorem \ref{caseoffixedmetric} shows that the original gwistor space structure we
found, the standard $\sigma_0$, is indeed preferred; it has greater interest than
the others on the circle (of course, besides the antipodal of $\sigma_0$, a
duality which as explained in section \ref{VoG2s} we shall not explore here).

We shall now see a result concerning the type of $\dx\sigma$ with respect to the
$G_2$-de\-com\-position of $\Lambda^4T^*\calG$. We follow the description by
\cite{FerGray} also found in several good references such as
\cite{Agri,Bryant1,Bryant2}. A structure is said to be of pure type $W_3$ if
$\dx\sigma=*\tau_3$ with $\tau_3$ the $W_3$ part, that is satisfying
$\tau_3\wedge\sigma=\tau_3\wedge*\sigma=0$.
\begin{teo}
 The gwistor space $(\calG,\sigma)$ of a constant sectional curvature $k$ manifold 
with $\sigma$
given as before and $f_0,f_1$ constant, is of pure type $W_3$ if and only if
$k=-2$.
\end{teo}
\begin{proof}
Our invoked Riemann tensor satisfies
$R_{ijpq}=k(\delta_i^q\delta_j^p-\delta_i^p\delta_j^q)$ for a constant sectional
curvature metric (this is not a sign convention; it is a compatibility condition
between required tensors on $\calG$ and tensors on the base manifold). By definitions
in (\ref{calralpha},\ref{calralpha1}), seen below but known from \cite{Alb2},
we have $\calR\alpha=-k\theta\wedge\alpha_1,\,\
\calR\alpha_1=-2k\theta\wedge\alpha_2$. 

Now, since the metric is Einstein we have $\dx*\sigma=0$ by Theorem
\ref{caseoffixedmetric} and thence
$\dx\sigma=\lambda*\sigma+*\tau_3$ (in other words, cf. \cite{Bryant2}, we have
$\tau_1=\tau_2=0$). The condition of pure type $W_3$, equivalently $\lambda=0\in\R$,
corresponds by a simple argument to $(\dx\sigma)\wedge\sigma=0$. 

With $\sigma=-f_0\alpha-f_1\alpha_1+f_0\alpha_2+f_1\alpha_3+\theta\wedge\dx\theta$,
we get the following formula:
\begin{equation}\label{dsigmaCSCcomSasaki}
 \dx\sigma=\theta\wedge\bigl(-3f_1\alpha+f_0(k+2)\alpha_1+f_1(2k+1)\alpha_2-
 3f_0k\alpha_3\bigr) +(\dx\theta)^2.
\end{equation}
Using the `first structure equations' from \eqref{bse1} or \cite[Proposition
2.1]{Alb2} and
$f_0^2+f_1^2=1$, we have
\begin{eqnarray*}
 \dx\sigma\wedge\sigma &=& (3f_1^2+3f_0^2(k+2)+3f_1^2(2k+1)+3f_0^2k+6)\Vol_\calG\\
  &=& (6f_1^2+6f_0^2+6(f_1^2+f_0^2)k+6)\Vol_\calG\\
  &=& 6(2+k)\Vol_\calG .
\end{eqnarray*}
Hence the result.
\end{proof}
We recover, in particular, the result in \cite[Corollary 3.1]{Alb2} for the
preferred $\sigma_0=\alpha_2-\alpha+\theta\wedge\dx\theta$ on hyperbolic space of
sectional curvature $-2$. Notice however the independency from the pair $(f_0,f_1)\in
S^1$. The same is true with the following quite noticeable formula.
\begin{prop}
Assuming the same conditions as above,  $\|\dx\sigma\|^2=48$ if and only if $k=-2$ or
$k=1$.
\end{prop}
\begin{proof}
 We easily deduce from \eqref{dsigmaCSCcomSasaki} that $\|\dx\sigma\|^2=12(k^2+k+2)$.
\end{proof}


\subsection{Properties of the general case}

Let us consider some metric problems related with the variations of gwistor space.

Suppose $(f_0,\ldots,f_4):\calG\rr\R^5$ is a function satisfying the conditions in
Theorem \ref{condicoesdeestabilidade}. We study those 3-forms
\begin{equation}\label{sigmageral2}
 \sigma=f_0\alpha+f_1\alpha_1+f_2\alpha_2+f_3\alpha_3+f_4\theta\wedge\dx\theta
\end{equation}
which define $G_2$-structures on $\calG\rr M$.

\begin{trivlist}
\item {\bf Remarks}. 1. Recall a metric almost contact structure is said to be
K-contact if the characteristic
vector field is Killing. In the case of the Sasaki metric, $(\calG,\theta,B^tU)$ is
K-contact if and only if $M$ is locally isometric to $S^4$ of radius 1, a result due
to
Y.~Tashiro. In general, our metrics $\langle\,\cdot\,,\,\cdot\,\rangle_\sigma$
induced from  $\sigma$ turned out to be `$g$-natural' contact metrics in
the sense of e.g. \cite{AbbaCal} (in particular the immediate question of
$\langle\,\cdot\,,\,\cdot\,\rangle_\sigma$ being K-contact is solved in the same
reference). 2.
Another feature of gwistor theory is that $\sigma$ seems to be never preserved by the
vector field $B^tU$. This is known both as the geodesic spray or the geodesic flow
vector field, cf. \cite{Geig,Sakai}. Indeed, computations for constant $f_i$ have
shown that the equation ${\cal L}_{B^tU}\sigma=0$ has no solution
$\sigma\in\Lambda_+^3$. For any $f_i$ defined on $\calG$, or even just the pull-back
of functions on $M$, one may write interesting differential equations.
\end{trivlist}

Now we shall compute the exterior derivatives of the $G_2$-structures. From the
formulas
in (\ref{ortcoframenewmetric1}) we deduce
\begin{equation}
 \theta=\frac{1}{f_4t^\frac{1}{2}}\tilde{\theta},\qquad\qquad 
 \dx\theta=\frac{1}{th^\frac{1}{2}}\widetilde{\dx\theta},\qquad\qquad 
\alpha=\frac{x^\frac{3}{2}}{(th)^\frac{3}{2}}\tilde{\alpha},
\end{equation}
\begin{equation}
\alpha_1=\frac{x^\frac{1}{2}}{t^\frac{3}{2}h}\bigl({\tilde{\alpha}}_1-
\frac{z}{h^\frac{1}{2}}\tilde{\alpha}\bigr),\qquad\qquad
\alpha_2=\frac{1}{x^\frac{1}{2}(th)^\frac{3}{2}}(h{\tilde{\alpha}}_2
-2h^\frac{1}{2}z\tilde{\alpha}_1+3z^2\tilde{\alpha}),
\end{equation}
\begin{equation}
 \alpha_3=
 \frac{1}{(txh)^\frac{3}{2}}\bigl(h^\frac{3}{2}\tilde\alpha_3-hz{\tilde{\alpha}}_2 
+h^\frac{1}{2}z^2\tilde{\alpha}_1-z^3\tilde{\alpha}\bigr).
\end{equation}
The forms with a tilde are defined algebraically using the orthonormal basis for
$\sigma$,
formally introduced as the respective $\theta,\dx\theta,\alpha,\ldots,\alpha_3$. For
instance
$\tilde{\theta}=\tile^0,\ \widetilde{\dx\theta}=\tile^{41}+\tile^{52}+\tile^{63}$,
cf. \eqref{algumasformas}. In
particular, we note that we may use the already mentioned `first
structure equations' from \eqref{bse1} but with a tilde!

We also need the inverse formulas of the above:
\begin{equation}
\widetilde{\theta\wedge\dx\theta}=f_4t^\frac{3}{2}h^\frac{1}{2}\theta\wedge\dx\theta,
\qquad\qquad\quad
\tilde\alpha=\frac{(th)^\frac{3}{2}}{x^\frac{3}{2}}\alpha,
\end{equation}
\begin{equation}
\tilde{\alpha}_1=\frac{ht^\frac{3}{2}}{x^\frac{3}{2}}\bigl(x\alpha_1+3z\alpha\bigr),
\qquad\qquad
\tilde{\alpha}_2=\frac{h^\frac{1}{2}t^\frac{3}{2}}{x^\frac{3}{2}}\bigl(x^2\alpha_2+
2xz\alpha_1+3z^2\alpha\bigr),
\end{equation}
\begin{equation}
 \tilde\alpha_3=\frac{t^\frac{3}{2}}{x^\frac{3}{2}}\bigl(x^3\alpha_3+x^2z\alpha_2 
+xz^2\alpha_1+z^3\alpha\bigr).
\end{equation}
Using the `first structure equations' for the Hodge
operator of the metric and orientation induced by $\sigma$, and writing back in terms
of
the usual frame, we obtain the following result.
\begin{teo}\label{asteriscosigmabasics}
\begin{equation}\label{asteriscosigmabasicas1}
 *_\sigma(\theta\wedge\dx\theta)=\frac{t^\frac{1}{2}h^\frac{1}{2}}{2f_4}(\dx\theta)^2
,
\end{equation}
\begin{equation}\label{asteriscosigmabasicas2}
*_\sigma\alpha=\frac{f_4t^\frac{1}{2}}{h^\frac{3}{2}}
\theta\wedge\bigl(x^3\alpha_3+x^2z\alpha_2+xz^2\alpha_1+z^3\alpha\bigr),
\end{equation}
\begin{equation}\label{asteriscosigmabasicas3}
*_\sigma\alpha_1=-\frac{f_4t^\frac{1}{2}}{xh^\frac{3}{2}}
\theta\wedge\bigl(3x^3z\alpha_3+x^2(h+3z^2)\alpha_2+x(2hz+3z^3)\alpha_1
+(3hz^2+3z^4)\alpha\bigr),
\end{equation}
\begin{equation}\label{asteriscosigmabasicas4}
*_\sigma\alpha_2=\frac{f_4t^\frac{1}{2}}{x^2h^\frac{3}{2}}
\theta\wedge\bigl(3x^3z^2\alpha_3+x^2(2hz+3z^3)\alpha_2+x(h^2+4hz^2+3z^4)\alpha_1
+(3h^2z+6hz^3+3z^5)\alpha\bigr),
\end{equation}
\begin{equation}\label{asteriscosigmabasicas5}
\begin{split}
*_\sigma\alpha_3=-\frac{f_4t^\frac{1}{2}}{x^3h^\frac{3}{2}}
\theta\wedge\bigl(x^3z^3\alpha_3+x^2(hz^2+z^4)\alpha_2+ \bigr.  \hspace{5cm}\\
\bigl. +x(h^2z+2hz^3+z^5)\alpha_1+(h^3+3h^2z^2+3hz^4+z^6)\alpha\bigr).
\end{split}
\end{equation}
\end{teo}
\begin{coro}
The Hodge $*$ operator is homogeneous of degree $\frac{1}{3}$ on 3-forms
viewed as a map $\sigma\rightsquigarrow*_\sigma$.
\end{coro}
\begin{proof}
From definitions, we see $x,y,z$ have degree 2 and thence $h$ has degree 4;
then $m=f_4h^\frac{1}{3}$ and $\Vol_\sigma$ have degree $\frac{7}{3}$
and finally
$t=f_4/m$ has degree $-\frac{4}{3}$. Finally, observing
(\ref{asteriscosigmabasicas1}) the result follows (from the definition will also do).
\end{proof}

Now we recall the formulas from \cite[Proposition 2.3]{Alb2}:
\begin{equation}\label{derivadasdosalphas}
\begin{split}
&\dx\alpha=\calR\alpha, \qquad\dx\alpha_1=3\theta\wedge\alpha+\calR\alpha_1,\\
&\quad\dx\alpha_2=2\theta\wedge\alpha_1-\underline{r}\,\vol,\qquad
\dx\alpha_3=\theta\wedge\alpha_2 .
\end{split}
\end{equation}
$\calR\alpha,\calR\alpha_1$ are linearly independent forms depending on the
curvature $R$ of $M$, and $\underline{r}$ is a scalar function on $\calG$ defined by
$\underline{r}(u)=r(u,u)$, with $R$ and $r$ the usual Riemann and Ricci
curvature tensors. Concretely, cf. \cite[formulas 25 and 26]{Alb2},
\begin{equation}\label{calralpha}
 \calR\alpha=\sum_{0\leq i<j\leq3}R_{ij01}e^{ij56}+R_{ij02}e^{ij64}+R_{ij03}e^{ij45},
 \end{equation}
\begin{equation}\label{calralpha1}
 \calR\alpha_1=\sum_{0\leq i<j\leq3}
R_{ij01}(e^{ij26}+e^{ij53})+R_{ij02}(e^{ij61}+e^{ij34})
+R_{ij03}(e^{ij15}+e^{ij42}).
\end{equation}
In particular $\theta\wedge\calR\alpha_1=-\rho\wedge\vol$ where
$\rho=\sum_{i=1}^3r(e_i,e_0)e^{i+3}$.
\begin{teo}\label{dsigmanevervanishes}
 For any functions $f_0,\ldots,f_4$, we have $\dx\sigma\neq0$. 
\end{teo}
\begin{proof}
Indeed, since $\dx\theta\wedge\alpha_i=0,\forall i=0,1,2,3,\ \alpha_0=\alpha$, we
have by the Bianchi identity
\begin{eqnarray*}
 \theta\wedge\dx\theta\wedge\dx\sigma &=&
\theta\wedge\dx\theta\wedge\bigl(f_4(\dx\theta)^2+\sum\dx
f_i\wedge\alpha_i+f_i\dx\alpha_i\bigr)\\
&=& (6f_4+f_0(R_{2301}+R_{3102}+R_{1203}))\Vol_\calG=6f_4\Vol_\calG .
\end{eqnarray*} 
However, we saw $f_4$ must be positive.
\end{proof}

From now on we assume the functions $f_0,\ldots,f_4$ are constant.

Returning to the Hodge duals of Theorem \ref{asteriscosigmabasics}, then we have by
simple reasons
\begin{equation}\label{ddasasteriscosigmastructureforms0}
 \dx(*_\sigma(\theta\wedge\dx\theta)) = 0,
\end{equation}
\begin{equation}
\dx(*_\sigma\alpha)= -\frac{f_4t^\frac{1}{2}}{h^\frac{3}{2}}
\theta\wedge\bigl(xz^2\calR\alpha_1+z^3\calR\alpha\bigr),
\end{equation}
\begin{equation}
\dx(*_\sigma\alpha_1) = \frac{f_4t^\frac{1}{2}}{xh^\frac{3}{2}}
\theta\wedge\bigl(x(2hz+3z^3)\calR\alpha_1+(3hz^2+3z^4)\calR\alpha\bigr),
\end{equation}
\begin{equation}
 \dx(*_\sigma\alpha_2) = -\frac{f_4t^\frac{1}{2}}{x^2h^\frac{3}{2}}
\theta\wedge\bigl(x(h^2+4hz^2+3z^4)\calR\alpha_1+(3h^2z+6hz^3+3z^5)\calR\alpha\bigr),
\end{equation}
\begin{equation}\label{ddasasteriscosigmastructureforms1}
\dx(*_\sigma\alpha_3) = \frac{f_4t^\frac{1}{2}}{x^3h^\frac{3}{2}}
\theta\wedge\bigl(x(h^2z+2hz^3+z^5)\calR\alpha_1+
(h^3+3h^2z^2+3hz^4+z^6)\calR\alpha\bigr) .
\end{equation}

Adding up the above with the respective coefficients from \eqref{sigmageral}, we find
the vanishing of the two polynomials
\begin{equation}\label{polinome1}
{\mathfrak{p}}_1=
-f_0x^3z^2+f_1x^2(2hz+3z^3)-f_2x(h^2+4hz^2+3z^4)+f_3(h^2z+2hz^3+z^5),
\end{equation}
\begin{equation}\label{polinome2}
{\mathfrak{p}}_2= 
f_0x^3z^3-f_1x^2(3hz^2+3z^4)+f_2x(3h^2z+6hz^3+3z^5)-f_3(h^3+3h^2z^2+3hz^4+z^6)
\end{equation}
is a sufficient condition for the vanishing of $\dx(*_\sigma\sigma)$:
\begin{equation}
 \dx(*_\sigma\sigma)=\frac{f_4t^\frac{1}{2}}{x^3h^\frac{3}{2}}
\theta\wedge\bigl(x{\mathfrak{p}}_1\,\calR\alpha_1-{\mathfrak{p}}_2\,
\calR\alpha\bigr).
\end{equation}
Also the reader understands now why we chose constant coefficients. If
$z\neq0$, we may multiply the first polynomial by $z$, add to the second and
factor out a $h(>0)$ from the result, to obtain:
\begin{equation}\label{polinome21}
  -f_1 x^2 z^2 + 2 f_2 x h z + 2 f_2 z^3 x - f_3 h^2 - 2 f_3 h z^2 - f_3 z^4.
\end{equation}
Finally, introducing equations (\ref{defsdetxyz},\ref{defdeh}) and resorting to
some computer algebra software, we are able to find two independent
expressions in the original parameters $f_0,\ldots,f_3$:
\begin{equation}\label{polinome1T}
{\mathfrak{p}}_1= -f_0 \left(f_1^2-f_0 f_2\right) \left(-f_2^2+f_1 f_3\right)^2
\end{equation}
\begin{equation}\label{polinome2T}
\begin{split}
&{\mathfrak{p}}_2= (f_2^2 - f_1 f_3)^3 \bigl(-2 f_0 f_1^3 f_2^3 + 3 f_0^2 f_1 f_2^4
-f_1^6 f_3 + 6 f_0 f_1^4 f_2 f_3 -  6 f_0^2 f_1^2 f_2^2 f_3 \bigr.\\
& \hspace{4.6cm} \bigl. - 2 f_0^3 f_2^3 f_3 - 
   3 f_0^2 f_1^3 f_3^2 + 6 f_0^3 f_1 f_2 f_3^2 - f_0^4 f_3^3\bigr)
\end{split}
 \end{equation}
Notice they are homogeneous, as expected, and notice the factor $y=f_1^2-f_0f_2$ in
the
second polynomial and the common factor $x=f_2^2-f_1f_3$, which must both be positive
by
hypothesis. From equivalence we get the simple expression 
\begin{equation}\label{polinome21T}
(f_1^3-2 f_0 f_1 f_2+f_0^2 f_3)(f_2^2 - f_1 f_3)^3\ \ (=(\ref{polinome21})).
\end{equation}
\begin{teo}\label{cocalibrationingeneral}
A 3-form $\sigma$ as above defining a $G_2$-structure, with $f_0,\ldots,f_4$
constant, satisfies
$\dx*_\sigma\sigma=0$ if and only if any one of the following occurs:\\
(i) the polynomial $\mathfrak{p}_2$ from (\ref{polinome2T}) vanishes and  $M$ is
Einstein.\\
(ii) $M$ has constant sectional curvature.
\end{teo}
\begin{proof}
Notice first that $\dx*_\sigma\sigma=0$ if and only if both $\theta\wedge
{\mathfrak{p}}_1\,\calR\alpha_1$ and $\theta\wedge{\mathfrak{p}}_2\,
\calR\alpha$ vanish.
Also we note that, if $f_0=0$, then neither $f_1$ or $f_3$ can vanish (otherwise we
would get $y=0$ or $h=0$ from definition). 
So the two main polynomials cannot vanish simultaneously, as we see directly, or from
the
implied equation (\ref{polinome21T}).

Now, if ${\mathfrak{p}}_2$ vanishes, then we may conclude that $f_0\neq0$,
ie. the first polynomial ${\mathfrak{p}}_1$ does not vanish. So the cocalibration
equation becomes equivalent to the vanishing of
$\theta\wedge\calR\alpha_1=-\rho\wedge\vol$, which happens if and only if $M$ is
Einstein. Conversely, if the polynomial ${\mathfrak{p}}_2$ does not
vanish, then the equation relies on a metric such that $\theta\wedge\calR\alpha=0$;
equivalently, $R_{1201}=R_{2301}=0$, etc. This is the same as $M$ having constant
sectional curvature. In particular, $M$ being Einstein.
\end{proof}
For example, if $f_0=0$, then we are certainly bound to the second case. 

Noteworthy is the case when $f_1f_2=f_0f_3$ (or $z=0$), which generalizes Proposition
\ref{caseoffixedmetric}. By formulas
(\ref{ddasasteriscosigmastructureforms0}...\ref{ddasasteriscosigmastructureforms1})
we see
\begin{equation}
 \dx*_\sigma\sigma=f_3\frac{f_4t^\frac{1}{2}}{x^3h^\frac{3}{2}}
h^3\theta\wedge\calR\alpha=
\frac{f_3f_4t^\frac{1}{2}h^\frac{3}{2}}{ x^3}\theta\wedge\calR\alpha.
\end{equation}

A question put to the author by colleagues was: if we could always find,
invariant of the metric on $M$, a natural $G_2$ structure which would be co-closed.
The answer is no, because the two polynomials do not vanish simultaneously. By the
contrary we stress the relevance of $G_2$ cocalibration goes much beyond the known
cases and examples.

\subsection{Nearly-parallel $G_2$-structures}

Nearly-parallel $G_2$-structures on 7-dimensional manifolds are defined by
$\dx*_\sigma\sigma=0$ and
$\dx\sigma=c*_\sigma\sigma$ for some constant $c$. Clearly, if $c\neq0$, the
condition is simply the latter equation.

We consider a variation of the $G_2$ structure on $\calG$, as in (\ref{sigmageral2}). In order to
find a nearly-parallel structure $\sigma$, we may assume already that it is cocalibrated
($c\neq0$). Recall the Hodge $*$ operator is homogeneous of degree $1/3$ on 3-forms
viewed as a map $\sigma\rightsquigarrow*_\sigma$. Hence if we find a solution
to the above in our subspace of $\sigma\in\Lambda^3_+$, we find a line of solutions:
\begin{equation}
 \dx(s\sigma)=cs*_\sigma\sigma=cs^{-\frac{1}{3}}*_{s\sigma}s\sigma,\quad
s\in\R^+.
\end{equation}

We restrict here to the case $z=f_1f_2-f_0f_3=0$, the less `prohibitive' condition. And
continue to assume the coefficients are constants.
\begin{teo}
Under the previous condition, the only metric on an oriented Riemannian 4-manifold $M$ for
which a $(\calG,\sigma)$ is nearly-parallel is the constant sectional curvature 1 metric. Then
there are two classes of solutions, represented by the following two $G_2$-structures:
\begin{equation}
 \sigma_\pm=\pm\frac{\sqrt{2}}{2}(\alpha_2-\alpha+\alpha_3-\alpha_1)+\sqrt{\frac{3}{2}}
\theta\wedge\dx\theta,
\end{equation}
both satisfying $\dx\sigma=\sqrt{6}*_\sigma\sigma$.
\end{teo}
\begin{proof}
Since we assume $z=0$ and this is maintained on the line $\R^+\sigma$, there exists a positive
multiple of $\sigma$ such that $(f_0,f_1)$ is in the unit circle. Then we easily deduce $x=y=1$ and
$f_2=-f_0,\
f_3=-f_1$. Hence $h=1=t$ and $m=f_4$, cf. (\ref{efecinco}).

From formulas (\ref{asteriscosigmabasicas1}...\ref{asteriscosigmabasicas5}) and the hypothesis of
$\sigma$ being nearly-parallel, we see the 4-form $\dx\sigma$ is again $SO(3)$-invariant. Then
we easily deduce the curvature restriction: it must be of the constant kind. The equation
$\dx\sigma=c*_\sigma\sigma$ is solved using those same formulas. Looking at
components, we find a system ($k$ is the sectional curvature)
\[ \left\{ \begin{array}{l}
  c=2f_4\\ f_0f_1-kf_0^2=0 \\ 2f_0f_1k+f_0f_1-3f_1^2=0 \\
  3f_1-2f_0f_4^2=0 \\ 2f_0+kf_0-2f_0f_4^2=0
 \end{array}\right.  . \]
This yields $f_0=f_1$, which occurs twice in the circle; and $k=1,\ f_4=\sqrt{3/2},\ c=\sqrt{6}$.
The given 3-forms satisfy the equation and are genuine $G_2$-structures.
\end{proof}
Notice the metric on $\calG$ is the same on both solutions. Now we recall the classification of
nearly-parallel $G_2$ structures in \cite{FriKaMoSe}. The ones we got correspond to
the Stiefel manifold $V_{5,2}=SO(5)/SO(3)$ in their Table 2, which is of course the unit tangent
sphere bundle of $S^4$. The $G_2$ structure is constructed as a $U(1)$-bundle over
the complex quadric
$G_{5,2}$, the Grassmannian of 2-planes, with a K\"ahler-Einstein metric. The resulting
nearly-parallel $G_2$ structure is said to be Einstein-Sasakian for {\em some}
homogeneous
$SO(5)$-invariant metric. We have thus found more detail of this case. It is also most
interesting to see that our result gives a metric coinciding precisely with the Einstein
metric on $V_{5,2}$ deduced in \cite[Theorem 4]{AbbaKowa}. It has Riemannian scalar
curvature $\frac{63}{4}$, by a formula there.


\end{document}